\documentclass[11pt]{article}
\usepackage[final]{graphicx}
\usepackage[cmex10]{amsmath} 
\usepackage{amssymb}
\usepackage{amsfonts}
\usepackage{amsthm}
\usepackage{setspace}
\usepackage[numbers,sort]{natbib}
\usepackage[final,colorlinks]{hyperref}
\usepackage[margin=1in]{geometry}
\usepackage{epstopdf}

\DeclareMathOperator{\E}{\mathbb{E}}

\DeclareMathOperator{\Geometricdist}{Geometric}

\DeclareMathOperator{\expdist}{Exp}
\DeclareMathOperator{\nbdist}{NB}

\newtheoremstyle{nonum}{}{}{\itshape}{}{\bfseries}{.}{ }{\thmnote{#3}}
\theoremstyle{nonum}
\newtheorem{thm}{}
\newtheorem{lemma}{}

\title{Speeding Up Network Simulations Using Discrete Time}
\author{Aaron Lucas \and Benjamin Armbruster\footnote{%
Both authors contributed equally.}}
\date{\today}

\begin{document}
\maketitle

\begin{abstract}
We develop a way of simulating disease spread in networks faster at the cost of some accuracy.  Instead of a discrete event simulation (DES) we use a discrete time simulation.  This aggregates events into time periods.  We prove a bound on the accuracy attained.  We also discuss the choice of step size and do an analytical comparison of the computational costs.  Our error bound concept comes from the theory of numerical methods for SDEs and the basic proof structure comes from the theory of numerical methods for ODEs.
\end{abstract}

\section{Introduction} 
Traditional models for analyzing the spread of an infectious disease in a population rely on faulty assumptions of human behavior, such as random mixing, homogeneity and the lack of persistence of partnerships \cite{HazhirRahmandadJohn1978,bansal2007individual}. Network simulation models, which remedy these drawbacks, are at the forefront of health policy research in analyzing the control and prevention of infectious diseases \cite{kretzschmar2000sexual,morris1997concurrent}. Traditional compartmental models aggregate individuals in a population into different bins and only count the total number of individuals in certain states; whereas, network models track each individual's state separately. The advantage of tracking individuals separately instead of as groups allows the modeler to pinpoint  the origin of an outbreak and to target individuals with highly specific characteristics. Network models also have the unique ability to model network-based interventions such as contact tracing, partner delivered therapy or concurrency reduction \cite{armbruster2007contact,golden2005expedited,enns2011assessing}. 

Network models can range from simplistic and static to complex and dynamic \cite{enns2011assessing,pastor2001epidemic}. Underlying every network model is the notion of nodes and edges. Nodes usually represent individuals in the population, but can also represent cities or places \cite{hunter2008ergm,meyers2003applying,hufnagel2004forecast}. Edges connect the nodes and are defined for a duration of time that sufficient transmission of the infectious agent is possible if one of the nodes is infected.  In the simplest case there would only be one type of edge, but one can imagine using different types of edges to denote varying kinds of contact, such as sexual, friendship, etc. A simple network model of an infectious diseases would track over time the state of each node. The state of nodes would include the infection status of the node and might include other dynamic or stable endogenous or exogenous variables. 

Many network models are extensions of either an SI or SIS process. These contact processes are stylized models of disease spread where each node in the network is either in the susceptible or infected state with transmission occurring along edges.  Along with the random walk and the voter model, the contact process is one of the prototypical stochastic processes on networks. The SI model is for an incurable disease where once a node is infected it will remain so indefinitely while in the SIS model an infected node becomes susceptible after the infection has passed or is cured.  In this paper, we focus on SI and SIS processes on a network because of their foundational nature.  

Network models often need to be both large (in the number of nodes) and run for long periods of time to realistically model the long-term effects of an intervention in a large community.  In addition such simulations require much time (in replications or in clock time), in order to get a sufficiently precise estimate of the spread of an infection or the effects of different interventions. Indeed, we often need many replications of such a simulation to estimate the variability of the results, a quantity which is also important when designing large clinical trials or epidemiological studies to ensure that the results will be statistically significant.  

Given the computational effort involved in network modeling, it is natural to find ways to speed up the model efficiently with minimal side effects. A practical starting point is the taxonomy of models in \cite{Brennan2006}, which might point to a simpler type of model (e.g., a compartmental model) that may sometimes be appropriate.  Another, more mathematical approach to speeding up stochastic models is using a diffusion approximation.  These take a stochastic \emph{compartmental} model and take the limit as the population goes to infinity.  Using a functional central limit theorem, this limit becomes a stochastic differential equation \cite{Pollett2001}.  However, this approach does not work for network models since each node in our population is unique. Moreover, it is unclear how one would define the limit as the population grows while keeping the network structure constant.

We develop a different way of approximating an infection process in a population, one which is specific for network modeling.  Instead of using a natural discrete event simulation (DES) we use a discrete time simulation (DTS).  This aggregates events into time periods; for example with monthly time steps we would only update the infection status of nodes once a month instead of after each infection or cure.  Naturally, there is a loss in accuracy involved, especially for larger time steps.  With large time steps, a node infected in one time step might infect a second in the same time step, an event that would be captured in the DES but not in the DTS. We prove a strong convergence result and furthermore, a bound on the accuracy attained as a function of the time step.  This error bound allows us to do an analytical comparison of the computational costs for different step sizes and to estimate the maximum step size for the error that we can tolerate in our simulation. Superficially, our technique is similar to a method called tau-leaping, in that we aggregate events across time steps in a discrete event simulation to transform the model into a discrete time simulation \cite{Cao2006}. However, the similarities end there. In tau-leaping, only the number of individuals in each infection state is kept track of, similar to a compartmental model. In a network model, we care about each individual's position in the network. 

Bounding the error of a discrete time simulation of a continuous time process and proving a strong convergence result is a common issue in the theory of numerical methods for SDEs as the numerical methods are usually discrete time simulations \cite{higham2001algorithmic}.  However our stochastic process is not derived from Brownian motion as is the case for most SDEs.  Thus, while the concepts are similar, our proof approach is somewhat different.  
Our proof structure follows the convergence proofs for numerical methods for deterministic ODEs \cite{ascher1998computer}.

We assume our SI and SIS simulations on a network are Markov; that is, for the next state transition all we need to know is the current state of the network, or equivalently that the hazard rates are constant. In this way, we can formulate our DES as a continuous time Markov chain (CTMC) and our DTS as a discrete time Markov chain (DTMC). Thus, aside from the connections to numerical methods for SDEs and ODEs, we can also frame our problem as a way of constructing a DTMC from CTMC. However, our DTMC (representing our DTS) is only an approximation of the true DTMC achieved by sampling the CTMC at discrete time points. The true DTMC achieved by sampling is difficult to deal with since its transition matrix is the matrix exponential of the CTMC's rate matrix, whose size is already exponential in the number of nodes in the network.  Two other ways of constructing a DTMC from a CTMC are the embedded DTMC of a CTMC or the embedded DTMC of a uniformized CTMC  \cite{tijms2003first}.  However, these are not practical approximations. The embedded DTMC focuses on the sequence of states visited and loses all sense of ``when'' the transitions happen.  Also, the embedded DTMC has the same number of transitions as the CTMC.  The uniformized CTMC allows self-transitions in order to make all the sojourn times have the same expected duration. This leads it to have even more transitions in the same span of time as the original CTMC.  This defeats our goal of speeding up the simulation by reducing the number of transitions for a given time span.  
To our knowledge, this is the first rigorous analysis of the error of a DTS approximation of a DES \cite{buss2010comparison}.

In the next section we define our DES and DTS and then describe them from a Markov chain perspective. Then in section 3, we state the main results and give a sketch of the proofs.  The proof details are relegated to the appendix.  Section 4 illustrates these results numerically and section 5 analyzes the computational cost.  We conclude in section 6 with a discussion of future work.

\section{Model} 
In this section we describe the various parts of the model.  We start by introducing a DES and our network notation. Then we describe how our DTS ``batches'' events in our DES. We then define notation for the CTMC and DTMC stochastic processes, before precisely defining the SI and SIS contact processes. We end with a theorem stating an inequality between the DTMC and the CTMC for the SI process. Table 1 summarizes our notation.

\begin{table}
\centering
\begin{tabular}{p{6in}}
\hline
$n$, number of nodes\\
$E$, set of edges of the graph\\
$k$, max degree in the graph\\
$x$, generic state of the graph, a binary vector of length $n$\\
$x'\geq x$, the vector inequality is component-wise\\
$x_0$, initial state of the graph\\
$|\cdot|$, 1-norm of a vector\\
$|x|$, number of infected nodes when graph in state $x$\\
$S(x)$, set of susceptible nodes when graph in state $x$\\
$I(x)$, set of infected nodes when graph in state $x$\\
$n(j,x)$, number of (or set of edges from $j$ to) infected nodes neighboring $j$ when the graph is in state $x$ and $j$ is susceptible, and 0 (or $\emptyset$) otherwise\\
$1(\cdot)$, indicator function\\
$w \in \Omega$, scenario or sample point\\
DES, discrete event simulation, the CTMC\\
DTS, discrete time simulation, the DTMC\\
$\nbdist(r,p)$, a negative binomial random variable giving the number of successes before $r$ failures where the success probability is $p$.\\
$\Geometricdist(p)$, an geometric random variable on $\{0,1,2,\dotsc\}$ giving the number of failures before the first success where the success probability is $p$\\
$\expdist(\lambda)$, an exponential random variable with rate $\lambda$.\\
$Y_1 \leq  Y_2$ in distribution means $\Pr[Y_1\leq y] \geq  \Pr[Y_2\leq y]$.\\
$X(t)=X(t,w)$, true state (i.e., state of the CTMC/DES) at time $t$\\
$h$, length of each time step of the DTMC/DTS\\
$\tilde{X}_i(w) = X(ih,w)$, the CTMC sampled at time steps of length $h$\\
$(A_i)$, iid random vectors where $A_i$ contains the random numbers to simulate from time $(i-1)h$ to time $ih$\\
$(F_i)$, the natural filtration of $(A_i)$\\
$\tilde{g}(x,a)$, the transition function for $(\tilde{X}_i)$, $\tilde{X}_i=\tilde{g}(\tilde{X}_{i-1},A_i)$\\
$X_i(w)$, state of the DTMC/DTS approximation after $i$ steps, at time $ih$\\
$g(x,a)$, the transition function for $(X_i)$, $X_i=g(X_{i-1},A_i)$\\
$\epsilon_i=\epsilon_i(w)=X_i(w)-X(ih,w)$, global error\\
$f(x,a) = (g(x,a)-x)/h$, analogous to right hand side of an ODE\\
$D(x',x,a)=(x'-x)/h-f(x,a) = (x'-g(x,a))/h$, difference operator\\
$d_i=d_i(w)=D(X(ih,w),X((i-1)h,w),a)$, local error\\
$N(t)=N(t,w)=|X(t)|$, number of infected nodes in $X(t)$\\
$N_i=N_i(w)=|X_i|$, number of infected nodes in $X_i$\\
\hline
\end{tabular}
\caption{Notation and Acronyms}
\end{table}

\paragraph{Definition of a Discrete Event Simulation (DES)}
Discrete event simulations model an evolving system in continuous time as a sequence of events. Each event corresponds to a change in the system state. Events happen instantaneously and occur at separate time points. A set of timers are associated with each possible subsequent event. The time to next event is given by the smallest of the timers. When an event occurs, we say that the timer ``fires''. Immediately following an event, timers are updated and new possible events may be added. Conversely, in a discrete time simulation, time is divided into distinct intervals of length $h$. The status of the system changes instantaneously at the end of each interval. These updates approximate the cumulative changes that would have occurred during that interval in continuous time.

\paragraph{Graph Setup}
We consider a network (or graph) of $n$ nodes with set of edges $E$ and maximum degree $k$.  The state of the graph, $x\in\{0,1\}^n$, is an $n$-dimensional binary vector describing which nodes are infected; $x_i=1$ if node $i$ is infected.  For a graph in state $x$ and node $j$, we let $x+j$ denote the state where node $j$ is also infected.  We let $|x|$ be the 1-norm of the vector or equivalently the number of infected nodes.  We let $S(x)$ denote the set of susceptible nodes and $I(x)$ the set of infected nodes.  For graph states $x$ and $x'$, we let $x \leq  x'$ denote the component-wise inequality, that the infected nodes in $x'$ includes all those in $x$.  Depending on the context we let $n(j,x)$ be the set of edges between a susceptible node $j$ and its infected neighbors or the number of such neighbors (we define it to be 0 or $\emptyset$ if $j$ is not susceptible).

\paragraph{Batching}
Our approach to construct the DTS is to batch the events of the DES in a time step of length $h$.  Suppose at time $ih$ the network is in state $x$.  Now applying one step of the DTS is equivalent to selecting all the timers in the DES that are set to fire in time $h$ (i.e., before time $(i+1)h$) and executing them at once, essentially batching all the events that happen in that interval.  The updated state and set of timers is then ready for the next time step.  This is in contrast to the DES where we select only the next timer to fire and process the timers sequentially.

\paragraph{Definition of $X(t)$ and $X_i$}
Our DES and DTS of the infection processes on our graph can be formulated as a CTMC, $X(t,w)$, and a DTMC, $X_i(w)$, respectively. This is due to the Markov property that we assume for our SI and SIS processes. Here $w \in \Omega$ represents a scenario (a.k.a. sample point or state of the world).  Both processes start at the same initial state, $X(0)=X_0=x_0$. Our goal is to compare the CTMC $X(t)$ with the constructed DTMC $(X_i)$, so that we can compare the accuracy of our DTS with our DES.  The DTMC after $i$ steps, $X_i$, should approximate $X(ih)$. We will also define the DTMC $\tilde{X}_i(w)=X(ih,w)$ that denotes the CTMC sampled at the same time points.  Thus we will consider $X(t)$ and $(\tilde{X}_i)$ to be the true process and $(X_i)$ the DTMC approximation.

\paragraph{Definition of SI Process}
In the DES of the SI process, there are only transitions to states with one additional infected node, that is from a state $x$ we can only transition to states of the form $x+j$, where $j \in S(x)$. The transition rate from $x$ to $x+j$ is $n(j,x)$.  Nodes may not become uninfected: if node $j$ is infected in $X_i$ (or $\tilde{X}_i$), then $j$ is infected in $X_{i+1}$ (or (or $\tilde{X}_{i+1}$).  In the DTS, multiple nodes may become infected during a transition.  The probability of a node becoming infected in the DTS is as follows: for any susceptible node $j$ in $S(X_i)$, the probability of it being infected in $X_{i+1}$ is $1-\exp(-h n(j,X_i))$.  These probabilities are independent of each other.

\paragraph{Definition of SIS Process}
The SIS process differs from the SI process only in that it allows for recovery of infected nodes (and subsequent reinfection).  In addition to the transitions of the SI process, in the DES the graph may also transition from state $x+j$ to state $x$ when $j \in S(x)$.  The associated transition rate is $\mu$.  We construct the DTS is an analogous fashion as for the SI DTS. For any susceptible node $j$ in $S(X_i)$, the probability of it being infected in $X_{(i+1)}$ is again $1-\exp(-hn(j,X_i))$. For any infected node the probability of it being recovered in $X_{(i+1)}$ is $1-\exp(-\mu h)$. We also assumes, like in the SI DTS, that the transitions across nodes are independent of each other. That is, only primary infections may occur and any recovered node cannot become reinfected in the same time step. 

\paragraph{Filtration and DTMC}
It will be convenient to explicitly construct a filtration.  We let $(A_i)$ be a sequence of iid random vectors and let $(F_i)$ be the natural filtration of $(A_i)$.  We will assume that $\tilde{X}_i$ is adapted to $F_i$, that is $A_1,\dotsc,A_i$ contain all the information (i.e., all the random numbers) to simulate the DES until time $ih$.  For example we could let $A_i$ be a vector of iid exponential random variables, with the dimension of the vector large enough so that we have enough random numbers for any step of the simulation.  Specifically, $A_i$ would contain the random numbers for all the event timers in the DES from time $(i-1)h$ to time $ih$.  Thus we can denote the actions of the DES over a time interval of length $h$ by the function $x'=\tilde{g}(x,a)$ moving the graph state from $x$ to $x'$ with the random numbers in $a$, allowing us to define the $(\tilde{X}_i)$ recursively, $\tilde{X}_i=\tilde{g}(\tilde{X}_{i-1},A_i)$ for all $i\geq 1$.  Our goal is then to construct an easy to simulate DTMC $(X_i)$ adapted to the same source of random numbers, $(F_i)$.  Specifically, we seek to construct a simple function $x'=g(x,a)$, defining the approximate DTMC recursively, $X_i=g(X_{i-1},A_i)$ for all $i\geq 1$.

As an illustration, we now explicitly construct the filtration for the SI process.  We let each $A_i$ be a vector of $|E|$ iid $\expdist(1)$ random variables, one for each edge in the network, with $A_i(e)$ denoting the component corresponding to edge $e$.  If the graph is in state $x$ between time $(i-1)h$ and $ih$, then we have an active timer with time $A_i(e)$ for each edge $e$ between a susceptible and an infected node. This gives the above rate of infection, $n(j,x)$, for any susceptible node $j$.  In the DES, when a node $j$ becomes infected in the time interval $[(i-1)h,ih)$, we update the set of active timers using the components of $A_i$: we deactivate those active timers for edges between $j$ and another infected node and activate those for edges between $j$ and a susceptible node.  In the DTS, we do the same except that the timers are only updated at the end of each time step.

\begin{thm}[Theorem 1A]
For the SI process, $X(ih) \geq  X_i$ a.s.
\end{thm}
\begin{proof}See Appendix.\end{proof}

\section{Results and Proof Sketch} 

Our goal is to prove bounds on the global error, that is a strong convergence result.  The proofs are given later in this section.

\begin{thm}[Theorem 2A]
For the SI process, $\E[|\epsilon_i|] \leq  C_{SI} K_{SI} h$ when $h\leq 1$
where $C_{SI} = nk^2 e^{(k-2)}$, $K_{SI}=(1/k)(e^{kT}-1)$ and $T = ih$.
\end{thm}

\begin{thm}[Theorem 2B]
For the SIS process, $\E[|\epsilon_i|] \leq  C_{SIS} K_{SIS} h$ when $h\leq 1$
where $C_{SIS} = nk(k\exp(k-2)+\mu)$, $K_{SIS}=(1/(k+\mu))(e^{(k+\mu)T}-1)$ and $T = ih$.
\end{thm}

\paragraph{ODE Analogy}
We prove a rate of strong convergence using an analogy to ODEs: we treat the CTMC as an ODE of the form $\dot{x}=f(x)$ and treat the DTMC as Euler's method for this ODE.  Thus we define the right hand side function representing the incremental rate of change $f(x,a) = (g(x,a)-x)/h$, and we define the difference operator $D(x',x,a) = (x'-x)/h-f(x,a) = (x'-g(x,a))/h$.  The difference operator is designed so that for our approximation (i.e., Euler's method for ODEs and the DTMC in our case), $D(g(x,a),x,a)=0$ and thus the DTMC satisfies $D(X_i,X_{i-1},A_i)=0$ a.s. for all $i$.  While our argument holds with variable steps as in \cite{ascher1998computer}, we assume for simplicity that all time steps are of size $h$.  We define the local (i.e., 1-step) error as $d_i=D(X(ih),X((i-1)h),A_i)$ and the global (i.e., cumulative) error as $\epsilon_i=X_i-X(ih)$.  We first prove bounds on the local error.

\begin{lemma}[Lemma 3A]
For the SI process, $\E[|d_i|] \leq  C_{SI} h$ where $C_{SI}=n k^2 \exp(k-2)$ for $h\leq 1$.
\end{lemma}
\begin{proof}See Appendix.\end{proof}

\begin{lemma}[Lemma 3B]
For the SIS process, $\E[|d_i|] \leq  C_{SIS} h$ when $h\leq 1$
where $C_{SIS} = nk(k\exp(k-2)+\mu)$.
\end{lemma}
\begin{proof}See Appendix.\end{proof}

Note the dependence of the constants $C_{SI}$ and $C_{SIS}$ on $n$.  While this in undesirable it cannot be avoided since we accumulate error for each infected node.
Using the Lipschitz property of $f$ we prove 0-stability of Euler's method which then gives us the desired strong convergence.

\begin{lemma}[Lemma 4A]\label{lem4A}
For the SI process, $\E[|f(x,A_1)-f(z,A_1)|] \leq  L_{SI} |x-z|$ for all $x$ and $z$, where $L_{SI}=k$.
\end{lemma}
\begin{proof}See Appendix.\end{proof}

\begin{lemma}[Lemma 4B]\label{lem4B}
For the SIS process, $\E[|f(x,A_1)-f(z,A_1)|] \leq  L_{SIS} |x-z|$ for all $x$ and $z$, where $L_{SIS}=k+\mu$.
\end{lemma}
\begin{proof}See Appendix.\end{proof}

\begin{lemma}[Lemma 5]\label{lem5}
Euler's method is 0-stable, that is for any two sequences of random variables $(Y_i)$ and $(Z_i)$ adapted to $(F_i)$ with $Y_0=Z_0$,
\[ \E[|Y_i-Z_i|] \leq  K \max_{1\leq j\leq i} \E[|D(Y_j,Y_{j-1},A_j)-D(Z_j,Z_{j-1},A_j)|],\]
where $K=(1/L)(\exp(LT)-1)$, $T=ih$, and $L$ is the Lipschitz constant from Lemma 4 (i.e., \nameref{lem4A} or \nameref{lem4B} depending on the process).
Depending on the process we may write $K$ as $K_{SI}$ or $K_{SIS}$.
\end{lemma}
\begin{proof}
The proof is given in the Appendix but essentially follows \cite[p41]{ascher1998computer} using Lemma 4 for the Lipschitz property.
\end{proof}

We next prove Theorem 2 (we do not distinguish between Theorem 2A and Theorem 2B because the proofs are analogous).
 
\begin{proof}[Proof of Theorem 2]
Since $X_0=\tilde{X}_0=x_0$, we can substitute $Y_i=X_i$ and $Z_i=\tilde{X}_i$ into \nameref{lem5}. Note that $D(X_j,X_{j-1},A_j)=0$, and $D(\tilde{X}_j,\tilde{X}_{j-1},A_j)=d_j$. Thus, $\E[|\epsilon_i|] \leq K \max_{1\leq j\leq i} \E[|d_j|]$.  Applying Lemma 3 proves the claim.
\end{proof}

Note that our proofs (in particular those of Lemma 5 and Theorem 2) also work with variable step sizes, $(h_i)$, where we let $h=\max_{1\leq j\leq i} h_i$.

\section{Numerical Example} 
While this is a theoretical paper we included a numerical example to see how the DTS compares the DES in practice. Our two test graphs were a 30x30 toroidal lattice (i.e., one that wraps around on all four sides) and a small world random graph.  The small world graph was created by starting with the toroidal 30x30 lattice and then randomly distributing an extra 450 edges in such a way that every node has five edges.  At the start of each simulation we randomly infected 10\% of the nodes.  We considered both the SI and SIS processes with an infection rate of 1 and a recovery rate (for the SIS process) of 0.2. Figure 1 shows the distribution of the prevalence at time 1 for the DES and the DTS with step sizes of 0.01 and 0.0215, while Figure 2 shows the difference between the average prevalence at time 1 between the DES and DTS of different step sizes.  In each case we used 1500 replications.  In Figure 2, we show a line of unit slope on the log-log plot since the theory we developed above suggests that the error is linear in the size of the time step.

While the two figures describe the accuracy of the DTS, the following table compares the computational costs.  We again consider the same cases as in Figure 1 but in addition to the average difference in prevalence also look at the number of events (i.e., changes in the state of a node); the number of time steps; and the CPU time.  Of course for a DES, the number of time steps will equal the number of events, and for a DTS, it will equal the reciprocal of the step size (since we ran the simulation until time 1).  The ratio of events to time steps tells us how many events are batched each time step on average: about five per step for the smaller step and ten per step for the larger steps.  The difference in number of events between between the DTS and DES tells us the number of secondary events that are lost in the DTS because they occur in the same time step as the event that caused them (2--4\% depending on the step size).  While the simulation code is not optimized in any way, we nevertheless see speeds up 10x--20x faster for the DTS compared to the DES.  Again, we used 1500 replications for each case.  The last column tells us the difference in the mean prevalence at time 1.  We don't see the expected factor of two difference in the prevalence error between the two step sizes because the slope of line 1 in Figure 2 is not a perfect fit for the smaller step sizes. Nevertheless, even with the larger time steps the average difference in prevalence is less than 1.5 percentage points.

\begin{figure}[h]
	\centering
	\includegraphics[width=0.49\textwidth]{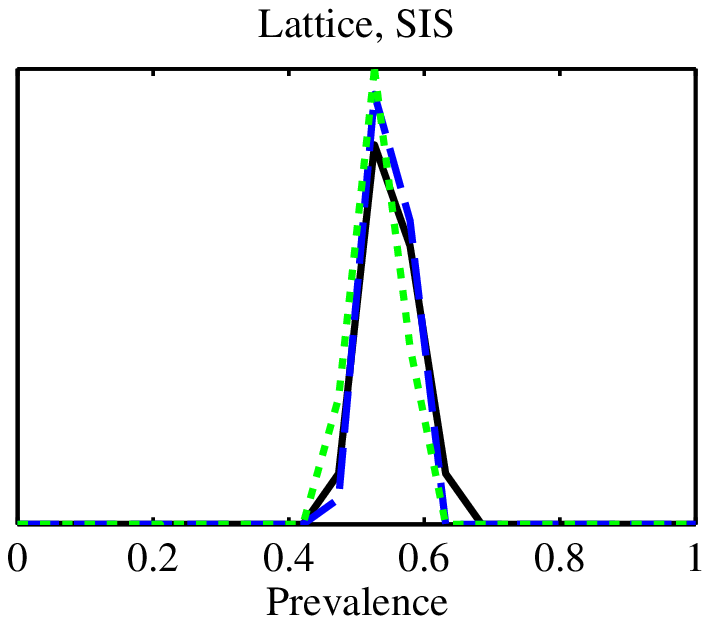}
	\includegraphics[width=0.49\textwidth]{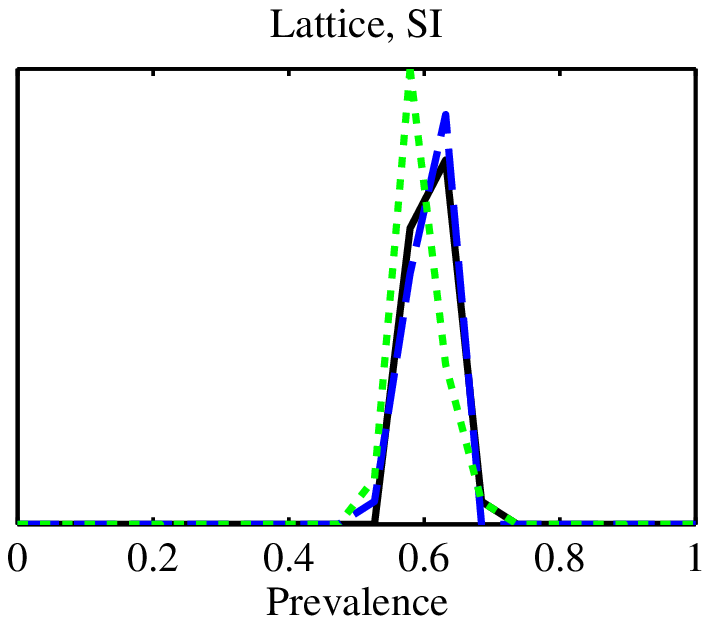}
	\includegraphics[width=0.49\textwidth]{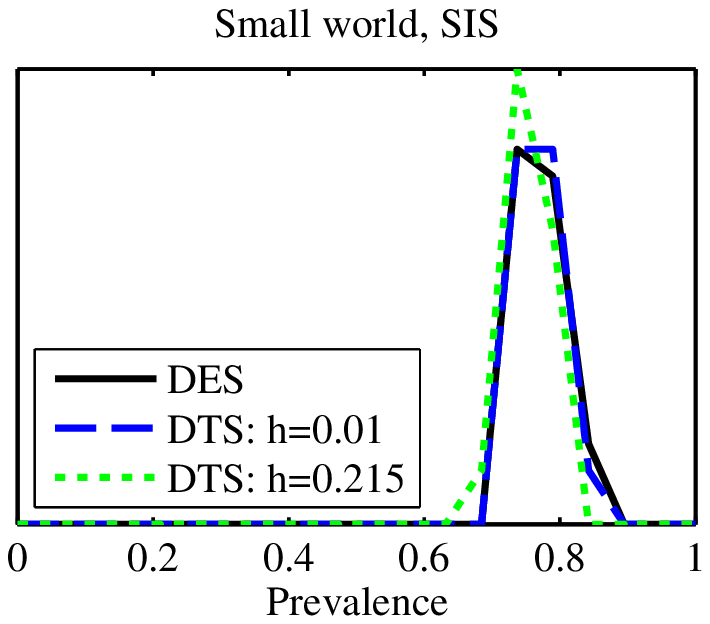}
	\includegraphics[width=0.49\textwidth]{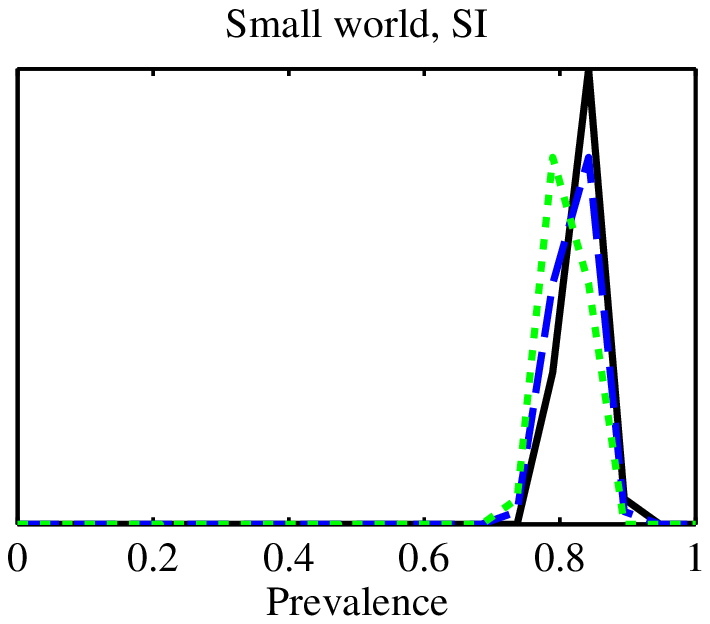}
	\caption{Probability density of the prevalence at time 1. We compare how well a DTS with two different step sizes approximates the DES.}
	\label{fig1}
\end{figure}

\begin{figure}[h]
	\centering
	\includegraphics[width=0.49\textwidth]{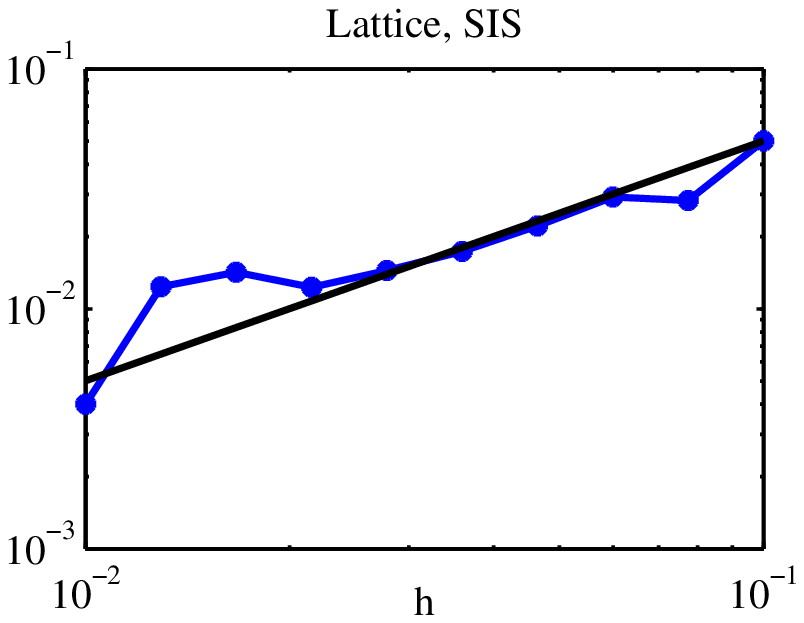}
	\includegraphics[width=0.49\textwidth]{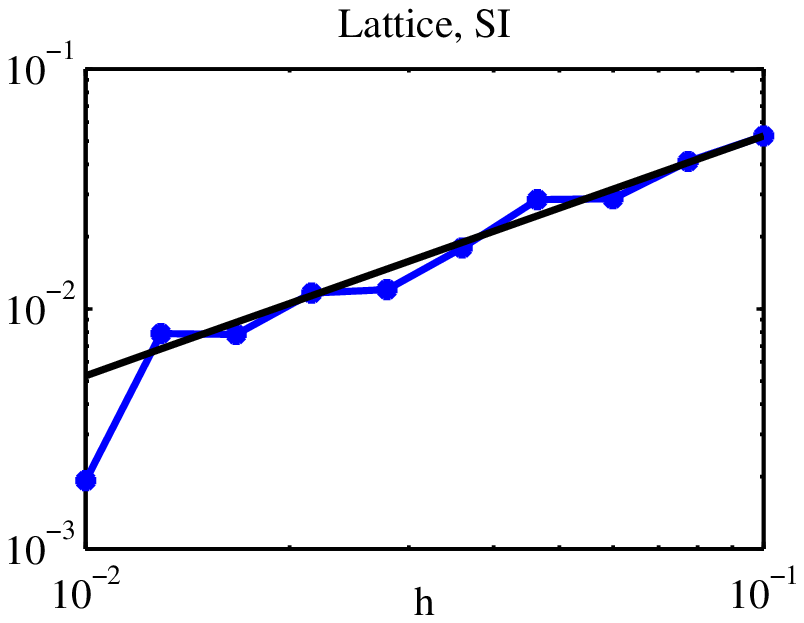}
	\includegraphics[width=0.49\textwidth]{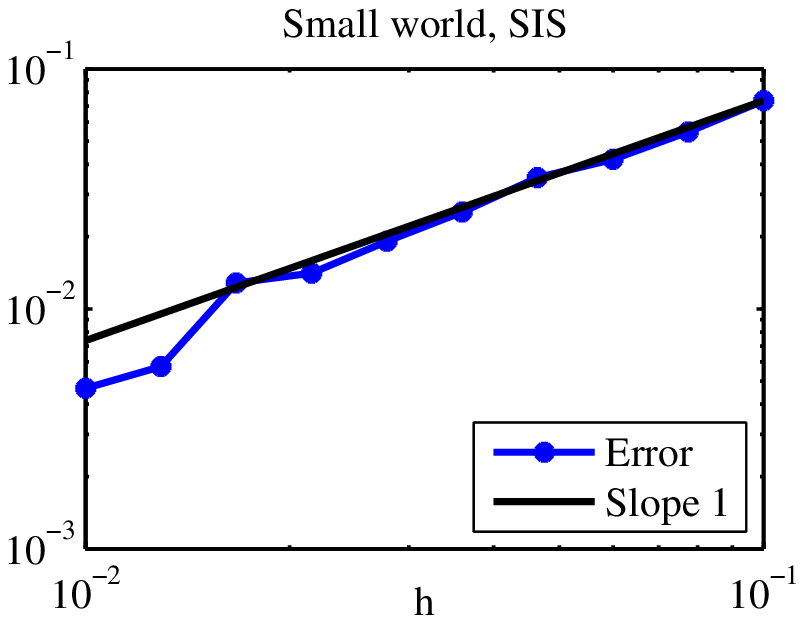}
	\includegraphics[width=0.49\textwidth]{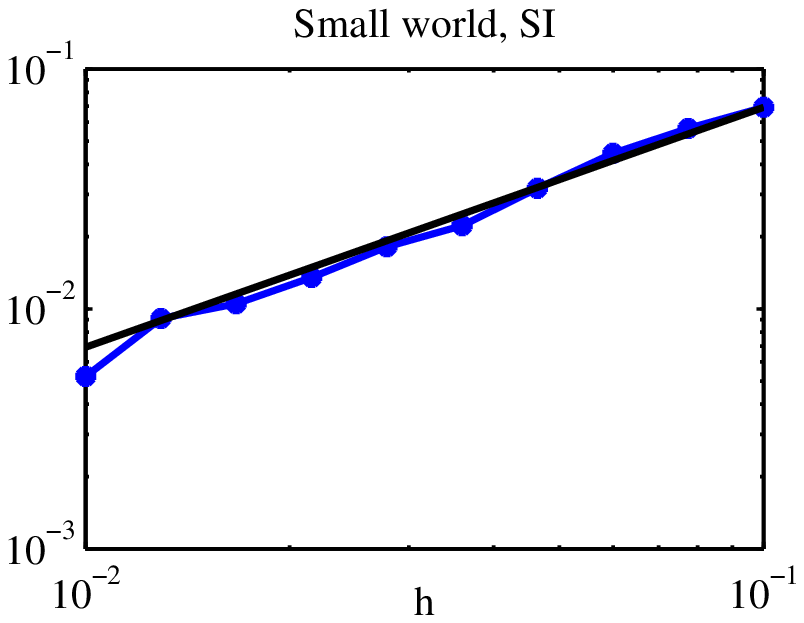}
	\caption{Absolute difference of the average prevalence at time 1 of the DES and the DTS.}
	\label{fig2}
\end{figure}

\begin{table}[h]
\centering
\begin{tabular}{lllrrrr}
\hline
Graph & Process & Algorithm & Events & Time Steps & CPU Time (s) & Prev. Diff.\\
\hline
Lattice & SIS & DES & 525.4 & 525.4 & 5.44 & \\
 &  & DTS: h=0.01 & 517.4 & 100.0 & 0.48 & 0.004\\
 &  & DTS: h=0.0215 & 505.1 & 46.0 & 0.27 & 0.012\\
\hline
Lattice & SI & DES & 464.6 & 464.6 & 3.16 & \\
 &  & DTS: h=0.01 & 461.9 & 100.0 & 0.48 & 0.002\\
 &  & DTS: h=0.0215 & 449.2 & 46.0 & 0.26 & 0.012\\
\hline
Small world & SIS & DES & 757.8 & 757.8 & 7.76 & \\
 &  & DTS: h=0.01 & 744.6 & 100.0 & 0.52 & 0.013\\
 &  & DTS: h=0.0215 & 731.6 & 46.0 & 0.29 & 0.014\\
\hline
Small world & SI & DES & 660.9 & 660.9 & 4.54 & \\
 &  & DTS: h=0.01 & 655.2 & 100.0 & 0.51 & 0.005\\
 &  & DTS: h=0.0215 & 644.0 & 46.0 & 0.28 & 0.014\\
\hline
\end{tabular}
\caption{Comparison of algorithms.}
\label{table1}
\end{table}

\section{Discussion of Computational Cost} 
To complement the numerical example of the previous section, we now provide a theoretical discussion of the computational cost of the DES and DTS. A DTS until time 1 with step size $h$ will have $1/h$ steps and differ from the DES in (i.e., have a global error of) $O(n)h$ nodes.  Dividing by $n$ we find that the difference in prevalence to the DES is $\delta=O(h)$.  Thus we can describe the number of time steps $O(1/\delta)$ in terms of the prevalence error.  Note how the number of time steps at a given level of accuracy does not depend on $n$, the size of the graph.  Nevertheless, the work in each time step will still depend on $n$.

In order to take this analysis further, to analyze the computational cost of the DES and the work per time step of the DTS, we turn to a concrete example.  Specifically, we will count the number of timers that need to be created/considered in the DTS versus the DES.  For example in an SI process on a tree, each node can be infected exactly one way, each edge being used once as a timer, and thus the computational cost is the same in the DES as in the DTS, for all step sizes.  This however is a special properties of trees, for general graphs, the DTS will use fewer timers than the DES.  One explanation is that by batching events in the DTS, some timers are not needed.  Specifically if nodes A and B share an edge and are infected in the same time step, then the timer for A infecting B is not required, while it would be required in the DES if A was infected first.  A network interpretation of this explanation is as follows.  Consider an SI process on a network where currently the nodes $I$ are infected and the nodes $S$ are susceptible.  For the next time step of the DTS, we create timers for all edges between $I$ and $S$.  If nodes $I'\subset S$ get infected in that step, then we never need to consider timers for infections (i.e., edges) between two nodes in $I'$, which we may need to for a DES, since in a DES, $I'$ always contains only a single node.

To make this more concrete, we turn to a random graph where every node has $k$ edges. (Such a graph can be created by placing edges uniformly at random among pairs of nodes, which don't already have degree $k$ or an edge between them.)  Considering the SIS process on this graph, we will estimate the computational costs.  Suppose our simulation starts near the steady-state prevalence, $p=|x_0|/n$.  Then for the next time step of the DTS we need approximately $|x_0|k(1-p)$ infection timers, one for each edge between $S$ and $I$. We ignore recovery events in this discussion because they differ less between the DTS and the DES. For the DES, there will be approximately $h|x_0|k(1-p)$ infection events in the next step of size $h$, each of which we can expect to create $k(1-p)$ additional timers.  Essentially creating $h|x_0|k^2(1-p)^2$ additional timers during the time step.  Thus the DES has a factor $kh(1-p)=O(\delta)$ more timers than the DTS in the same time period.  Note that the difference doesn't scale with $n$. This analysis does not find any economies or diseconomies of scale.

One place where there may be economies of scale for DTS and where our existing analysis is too conservative is when the network is composed of disjoint copies of a smaller network.  In that case, each disjoint subnetwork acts as an independent replication of the simulation.  In that case the central limit theorem applies and the difference in the number of infected nodes should scale with $\sqrt{n}$ instead of $n$, and thus the difference in prevalence should decrease as $1/\sqrt{n}$ instead of remaining constant.

\section{Conclusion and Future Work} 
In this paper we focused on network infection processes and proved error bounds for the accuracy of DTS as compared to DES.  We specifically focused on SI and SIS processes on networks with bounded degree and proved a strong convergence result where the expected difference in the number of infected nodes is proportional to the number of nodes and the step size.  This is the first such result.  It is also a result that brings together several diverse strands of research: DES, Markov chains, numerical methods for ODEs, dynamic processes on networks, and epidemiology. In section 5 we then demonstrate using a numerical example that this bound is linear with the step size and that DTS provide computational savings.

There are two directions for future work.  The practical direction would be to confirm the benefits of DTS using larger, less-stylized, simulations of disease spread, such as for example a simulation of HIV spread in an urban population of men who have sex with men.  The theoretical direction would be to extend these results to the more general stochastic processes found in the aforementioned practical simulations.  These are processes where nodes may have different susceptibilities to infection; progress to different infectious states (e.g., an incubation period or an acute stage of infection); and be put on treatment, essentially allowing nodes to be in more than just two states (susceptible and infected).  A useful framework for these more general stochastic processes are stochastic actor models \cite{Snijders2010} developed by social network scientists.  Methods have been developed to parameterize such models from data but ways to speed up their simulation have not.  Our contributions and such future improvements allow epidemiologists to quickly simulate disease spread on large population over decades to evaluate the efficacy of different intervention alternatives.

\bibliography{paperRefs}

\section{Appendix}

\begin{thm}[Theorem 1A]\label{thm1A}
For the SI process, $X(ih) \geq  X_i$ a.s.
\end{thm}
\begin{proof}
Since $X(0)=X_0$, it suffices to show that $\tilde{g}(x',a) \geq  g(x,a)$ for all $a$ whenever $x'\geq x$.  We prove this by showing for all $a$, that $\tilde{g}(x,a)\geq g(x,a)$ and that $g(x',a)\geq g(x,a)$ for $x'\geq x$. Now, it is easy to see that $\tilde{g}(x,a)\geq g(x,a)$ given $x$ since all the timers which fire in the DTS in time $h$ also fire in the DES. Finally, given $x'\geq x$, the nodes which are infected in $x$ are also infected in $x'$. Thus, any node which is infected in the next time step in $x$ also gets infected in $x'$, since we use the same set of random numbers $a$. 
\end{proof}

\begin{lemma}[Lemma S1]\label{lemS1}
$\E[\nbdist(r,p)] = r p/(1-p)$.
\end{lemma}
\begin{proof}Wikipedia or your favorite probability textbook.\end{proof}

\begin{lemma}[Lemma S2]\label{lemS2}
For $0\leq h\leq 1$ and $c\geq 0$, $\exp(c h)-1 \leq  h c \exp(c)$.
\end{lemma}
\begin{proof}
Using Taylor's theorem with remainder on $\exp(t)-1$ with $t\geq 0$, we have $\exp(t)-1 = t \exp(t')$ for some $t'\in [0,t]$. Thus $\exp(t)-1\leq t \exp(t)$.  Substituting $t=ch$ we have $\exp(ch)-1 \leq  h c \exp(ch)$.  Now, $h\leq 1$ and $c\geq 0$ imply $\exp(ch) \leq \exp(c)$, proving the claim.
\end{proof}


\begin{lemma}[Lemma S3]\label{lemS3}
For $0\leq a\leq b$, $(b-a)t-(\exp(-at)-\exp(-bt)) \geq  0$.
\end{lemma}
\begin{proof}
Note that $t+\exp(-t)$ is increasing for $t\geq 0$ because its first derivative is $1-\exp(-t)$ and $\exp(-t)\leq 1$ for $t\geq 0$.  Thus $(bt+\exp(-bt)) - (at+\exp(-at)) \geq  0$, proving the claim for $t\geq 0$.  Note that for $t\leq 0$, $t+\exp(-t)$ is decreasing because its first derivative is $1-\exp(-t)$ and $\exp(-t)\geq 1$ for $t\leq 0$.  Since $bt\leq at$ for $t\leq 0$ it follows that $(bt+\exp(-bt)) - (at+\exp(-at)) \geq  0$, proving the claim for $t\leq 0$.
\end{proof}

For notational convenience we define $N(t)=|X(t)|$ and $N_i=|X_i|$ to be the number of infected nodes in the DES and DTS respectively.

\begin{lemma}[Lemma S4]\label{lemS4} Consider the SI process on a network that is an infinite tree with the root node having $m$ children and all other nodes having $k-1$ children.  Suppose that initially, only the root node is infected, $N(0)=1$. Then $N(t)-N(0) \sim \nbdist(r,p)$ where $r=m/(k-2)$ and $p=1-\exp(-(k-2)t)$.
\end{lemma}
\begin{proof}
Each time a node is infected the number of edges between infected and susceptible nodes (i.e., the number of timers in the DES) increases by $k-2$ (the timer causing the infection is removed and $k-1$ timers are added due to the children of the newly infected node). Then assuming that $m=k-2$, the cumulative number of infections, $N(t)-N(0)$, is a a pure birth process where the hazard rate of an infection after $i$ infections is $\lambda_i = (k-2)i$, the so-called Yule process.  It is well known that the distribution of this process at time $t$ is $\Geometricdist(1-p)=\nbdist(1,p)$.
When $m\neq k-2$, we let $Z_1,\dotsc,Z_m$ denote the number of infections at time $t$ in each subtree of the root node.  Note that $N(t)-N(0)=Z_1+\cdots+Z_m$ and that the $Z_i$ are iid. Since the negative binomial distribution is divisible, $\nbdist(r_1,p)+\nbdist(r_2,p)\sim \nbdist(r_1+r_2,p)$, it follows from the $m=k-2$ case that $Z_1\sim \nbdist(1/(k-2),p)$.  Thus for general $m$, $N(t)-N(0)=Z_1+\cdots+Z_m\sim \nbdist(r,p)$.
\end{proof}

\begin{lemma}[Lemma S5]\label{lemS5} For the SI process, $N(t)-N(0) \leq  \nbdist(r,p)$ in distribution, where $r=N(0)k/(k-2)$ and $p=1-\exp(-(k-2)t)$.
\end{lemma}
\begin{proof}
To obtain an upper bound we maximize the hazard rate at every point in time.  Since the maximum degree in the network is $k$, the hazard rate at time 0 is at most $N(0)k$.  After a new infection, the hazard rate must decrease at least by one (from the removal of the timer of the newly infected node) and can increase at most $k-1$ (since the newly infected node has at most $k$ edges and one of those is to an already infected node).  Hence the hazard rate must increase no more than $k-2$.  The hazard rates from this upper bound are those of an SI process on an infinite tree with each node having $k-1$ children and the root node having $N(0)k$ children. Invoking \nameref{lemS4} proves the claim.
\end{proof}

\begin{lemma}[Lemma S6]\label{lemS6} For the SI process, $N(h)-N_1\leq \nbdist((N_1-N(0))r,p)$ in distribution, where $r=k/(k-2)$, $p=1-\exp(-h(k-2))$.
\end{lemma}
\begin{proof}
The infected nodes $X(h)$ at time $h$ are either those originally infected, $x_0$; those directly infected, $X_1-x_0$; or those subsequently infected, $X(h)-X_1$.  All the infected nodes in $X(h)-X_1$ stem (directly or indirectly) from infection by nodes in $X_1-x_0$ and were not directly infected by infected nodes in $x_0$ (otherwise the timers for such nodes would have fired by time $h$ implying that the nodes would be in $X_1$).  Thus as an upper bound on the number of infected nodes in $X(h)-X_1$ we can  look at all the infections caused by nodes in $X_1-x_0$ over a time period of length $h$.  Invoking \nameref{lemS5} with a graph initially in state $x_0'=X_1-x_0$ proves the claim.
\end{proof}

\begin{lemma}[Lemma S7]\label{lemS7}
If $\E[|f(x,A_1)-f(z,A_1)|] \leq  L |x-z|$ for all $x\geq z$ then $\E[|f(x,A_1)-f(z,A_1)|] \leq  L |x-z|$ for all $x$ and $z$.
\end{lemma}
\begin{proof}
Note that by the triangle inequality, $|f(x,a)-f(z,a)| \leq  |f(x,a)-f(\min(x,z),a)|+|f(z,a)-f(\min(x,z),a)|$ where $\min(x,z)$ is the component wise minimum of $x$ and $z$.  Then by our hypothesis, $\E[|f(x,A_1)-f(z,A_1)|] \leq  L |x-\min(x,z)| + L |z-\min(x,z)|$.  Since we use the 1-norm, $|x-\min(x,z)|+|z-\min(x,z)|=|x-z|$, proving the claim.
\end{proof}

\begin{lemma}[Lemma S8]\label{lemS8}
If $x\geq z$, then $\sum_j |n(j,x)-n(j,z)| \leq  k |x-z|$.
\end{lemma}
\begin{proof}
If we do not require that $n(j,x)=0$ when $j$ is infected, then proving the claim is simple.  In that case, $n(j,x)-n(j,z)\geq 0$ is the number of neighbors of $j$ that are infected in $x$ but not in $z$.  Hence, $\sum_j n(j,x)-n(j,z)$ is the sum of the degrees of the nodes infected in $x$ but not in $z$. There are $|x-z|$ such nodes, each of degree at most $k$, proving the claim.

However, we require a different proof since we define $n(j,x)=0$ when $x_j=1$. Suppose $x_j=0$.  Then since to $x\geq z$, $z_j=0$, and $n(j,x) \supseteq n(j,z)$. Hence, $|n(j,x)-n(j,z)|=|\{(j,j') : x_{j'}=1, z_{j'}=0\}|$.  Suppose $x_j=1$, and thus, $n(j,x)=0$. Then $|n(j,x)-n(j,z)|=n(j,z)$.  This equals 0 if $z_j=1$ and $|\{(j,j') : z_{j'}=1\}|$ if $z_j=0$.  Hence,
\begin{align*}
\sum_j |n(j,x)-n(j,z)| &=  |\{(j,j') : x_j=0, x_{j'}=1, z_{j'}=0\}| + |\{(j,j') : x_j=1, z_j=0, z_{j'}=1\}|\\
&= |\{(j,j') : x_j=0, x_{j'}=1, z_{j'}=0\}| \\
&\quad + |\{(j,j') : z_j=1, x_{j'}=1, z_{j'}=0\}| \qquad\text{switching $j$ and $j'$,}\\
&= |\{(j,j') : x_j=z_j=0, x_{j'}=1, z_{j'}=0\}| \\
&\quad + |\{(j,j') : x_j=z_j=1, x_{j'}=1, z_{j'}=0\}| \qquad \text{since $x_j\geq z_j$,}\\
&\leq  |\{(j,j') : x_{j'}=1, z_{j'}=0\}| \leq  k |x-z|.
\end{align*}

Here is an alternative proof. Note, $\sum_j |n(j,x) - n(j,z)| = $
\begin{align*}
 =& \sum_{j \in I(x)\setminus I(z)} n(j,z) + \sum_{j \in S(x)} |n(j,x) - n(j,z)|.\\
 \intertext{Since, $x\geq z$, when $j\in S(x)$, $n(j,x) \geq n(j,z)$ and thus,}
 =& \sum_{j \in I(x)\setminus I(z)} n(j,z) + \sum_{e_{ij} \in E} 1(i \in S(x), j \in I(x)\setminus I(z))\\
 =& \sum_{j \in I(x)\setminus I(z)} n(j,z) + \sum_{j \in I(x)\setminus I(z)} s(j,x)\\
 \intertext{where $s(j,x)$ is the number of susceptibles connected to $j$. Hence,}
\leq&  \sum_{j \in I(x)\setminus I(z)} n(j,z) + s(j,z)\\
\leq&  \sum_{j \in I(x)\setminus I(z)} k \leq |x-z|k.\\
\end{align*}
\end{proof}

\begin{lemma}[Lemma 3A]\label{lem3Aapx} For the SI process, $\E[|d_i|] \leq  C_{SI} h$ where $C_{SI}=n k^2 \exp(k-2)$ for $h\leq 1$.
\end{lemma}
\begin{proof}
Note that \[ d_i = (\tilde{X}_i - \tilde{X}_{i-1})/h - [g(\tilde{X}_{i-1},A_i) - \tilde{X}_{i-1}]/h
 = (\tilde{g}(\tilde{X}_{i-1},A_i)-g(\tilde{X}_{i-1},A_i))/h.\]
It suffices to bound $\E[|d_1|]$ for all $x_0$ to find a bound for $\E[|d_i|]$. Now, 
\begin{align*}
\E[|d_1|] &= \E[|\tilde{X}_1 - X_1|]/h\\
 &= \E[N(h) - N_1]/h \qquad \text{by \nameref{thm1A}}\\
 &\leq  \E[\nbdist((N_1-N(0))r,p)]/h \qquad \text{by \nameref{lemS6}},\\
\intertext{where $r=k/(k-2)$ and $p=1-\exp(-h (k-2))$.}
 &= (\E[N_1]-N(0))k/(k-2) (\exp((k-2)h)-1))/h \qquad \text{by \nameref{lemS1}}\\
 &\leq  (\E[N_1]-N(0)) k \exp(k-2) \qquad \text{by \nameref{lemS2}}.
\end{align*}
Note that 
\begin{align}
\E[N_1]-N(0)&=\sum_{j \in S(x_0)} 1-\exp(-h n(j,x_0)) \nonumber\\
& \leq  \sum_{j \in S(x_0)} h n(j,x_0)\nonumber\\
& \leq  \sum_{j \in S(x_0)} h k\nonumber\\
& \leq  nk h.\label{eq:N1N0}
\end{align}
Thus, $\E[|d_1|] \leq  n k^2 \exp(k-2) h$.
\end{proof}

Since $g(x,a)$ is in $\{0,1\}^n$, it follows that $f(x,a)$ is in $\{-1/h,0,1/h\}^n$ and hence $L=2n/h$ satisfies the inequality $\E[|f(x,A_1)-f(z,A_1)|] \leq  L |x-z|$.  In Lemma 4A and 4B, our goal is to prove this inequality with smaller values of $L$.  It will be convenient to denote the $j$th component of $f(x,a)$ and $g(x,a)$ by $f_j(x,a)$ and $g_j(x,a)$ respectively.  

\begin{lemma}[Lemma 3B] 
For the SIS process, $\E[|d_i|] \leq  C_{SIS} h$ when $h\leq 1$
where $C_{SIS} = nk(k\exp(k-2)+\mu)$.
\end{lemma}
\begin{proof}
As in the proof of \nameref{lem3Aapx}, it suffices to bound $\E[|\tilde{X}_1 - X_1|]/h$ for all $x_0$ to find the bound for $\E[|d_i|]$.  Note that the same subset of the original infected nodes, $I(x_0)$, have a recovery event in the DTS as in the DES because we use the same random numbers, $A_1$.  (We say recovery event, because we are not excluding the possibility that in the DES a node infected at time 0, recovers, and becomes reinfected by time $h$.)  Further, the original infected nodes, $I(x_0)$, directly infect the same set of nodes in the DES as in the DTS, the ones whose infection timers fire before time $h$.  (This is a similar argument as in \nameref{thm1A}.)  We define these directly infected nodes as $Y=g_{SI}(x_0,A_1)-x_0$, where we write $g_{SI}$ to distinguish the transition function for the SI process.

The difference $\tilde{X}_1-X_1$ then consists of the net additional infections caused by the directly infected nodes $Y$, and the recoveries among those directly infected.  (Note that reinfections of recovered nodes (including any in $I(Y)$ and $I(x_0)$) are counted among the additional infections.) We say \emph{net} additional infections because some of these additional infections might recover in the same time step.  Since the SI process does not allow for recoveries, the set of additional infections caused by those in $Y$ is at least as large for the SI process as for the SIS process.  (A precise way of saying this is that if $j\in I(Y)$ infects node $i$ via some sequence of intermediate nodes in the SIS process, then the same occurs in the SI process, since we use the same random numbers, $A_1$.) Hence, we can use \nameref{lemS6} to justify the same upper bound for the additional infections as  in the proof of \nameref{lem3Aapx} for the SI process: $\nbdist(Yr,p)$, where $r=k/(k-2)$ and $p=1-\exp(-h(k-2))$. For any infected node in $Y$, $\Pr[\expdist(\mu)\leq h]$ provides an upper bound on the probability of a recovery.  Thus,
\begin{align*}
\E[|\tilde{X}_1 - X_1|] &\leq \E[\nbdist(Yr,p)]+\E[Y]\Pr[\expdist(\mu)\leq h] \\
&=\E[Y](k/(k-2)(\exp((k-2)h)-1)+(1-\exp(-\mu h))) \qquad \text{by \nameref{lemS1}}\\
&\leq \E[Y](kh\exp(k-2)+\mu h) \qquad \text{by \nameref{lemS2}.}
\end{align*}
As we showed in \eqref{eq:N1N0} of \nameref{lem3Aapx}, $\E[Y]\leq nkh$. Thus, 
\[ \E[|d_1|]=\E[|\tilde{X}_1 - X_1|]/h \leq nk(k\exp(k-2)+\mu)h.\]
\end{proof}

\begin{lemma}[Lemma 4A]\label{lem4Aapx}
For the SI process, $\E[|f(x,A_1)-f(z,A_1)|] \leq  L_{SI} |x-z|$ for all $x$ and $z$, where $L_{SI}=k$.
\end{lemma}
\begin{proof}
Using \nameref{lemS7}, we may assume that $x\geq z$.  Let $y=f(x,A_1)-f(z,A_1)$.  We will show that $\E[|y_j|]\leq |n(j,x)-n(j,z)|$.  By \nameref{lemS8}, 
\[ \E[|f(x,A_1)-f(z,A_1)|] = \sum_j \E[|y_j|] \leq  \sum_j |n(j,x)-n(j,z)| \leq  k |x-z|.\]

Now we show that $\E[|y_j|]\leq |n(j,x)-n(j,z)|$. Note that if $x_j=1$, then $g_j(x,A_1)=1$ a.s. and thus $f_j(x,A_1)=0$ a.s. Hence, if $x_j=z_j=1$, then the inequality holds because $y_j=0$ a.s. If $x_j=1$ and $z_j=0$, then $n(j,x)=0$ and $y_j=g_j(z,A_1)/h$. Futhermore, $\E[g_j(z,A_1)]=1-\exp(-h n(j,z))\leq h n(j,z)$, thus proving the inequality in this case.  We now turn to the remaining case (remember $x_j\geq z_j$), where $x_j=z_j=0$.  In that case, $|y_j|=|g_j(x,A_1)-g_j(z,A_1)|/h$.  Then $g_j(x,A_1)=1(\min_{e \in n(j,x)} A_1(e) \leq  h)$.  Since $n(j,x)\supseteq n(j,z)$, it follows that $g_j(x,A_1)\geq g(z,A_1)$. Hence, $y_j=0$ if $g_j(z,A_1)=1$ or $g_j(x,A_1)=0$. In the remaining case, $g_j(x,A_1)=1$, $g_j(z,A_1)=0$, and $y_j=1/h$.  Thus, 
\begin{gather*}
 \Pr[y_j=1/h]=1-\Pr[g_j(z,A_1)=1 \text{ or } g_j(x,A_1)=0]\\
=1-(\Pr[g_j(z,A_1)=1]+\Pr[g_j(x,A_1)=0])
\intertext{since $g_j(x,A_1)\geq g(z,A_1)$ implies that these are exclusive events}
\Pr[y_j=1/h]=1-(1-\exp(-n(j,z)h)+\exp(-n(j,x)h))=\exp(-n(j,z)h)-\exp(-n(j,x)h).
\end{gather*}
Since $n(j,z)\leq n(j,x)$, we know from \nameref{lemS3} that $\Pr[y_j=1/h]\leq (n(j,x)-n(j,z))h$, proving the inequality for the case of $x_j=z_j=0$.
\end{proof}

\begin{lemma}[Lemma 4B]\label{lem4Bapx}
For the SIS process, $\E[|f(x,A_1)-f(z,A_1)|] \leq  L_{SIS} |x-z|$ for all $x$ and $z$, where $L_{SIS}=k+\mu$.
\end{lemma}
\begin{proof}
We take the same approach as in the proof for \nameref{lem4Aapx} and use \nameref{lemS7} to assume that $x \geq  z$. Again, we let $y_j = |f_j(x,A_1)-f_j(z,A_1)|$. This time will show that $\E[y_j] \leq  |n(j,x) - n(j,z)| + \mu |x_j-z_j|$. Then applying \nameref{lemS8} proves our claim:
\[ \E[|f(x,A_1) - f(z,A_1)|] \leq  \sum_j |n(j,x)-n(j,z)| + \mu |x_j-z_j| \leq  (k + \mu)|x-z|.\]
We now prove the inequality for $y_j$ by considering three cases (recall $x\geq z$):

Case 1: $x_j = z_j = 0.$ The reasoning for this case is the same as in \nameref{lem4Aapx}. Now $y_j = |g_j(x,A_1)-g_j(z,A_1)|/h \in \{0,1/h\}$. Now, $y_j = 0$ if $g_j(z,A_1) = 1$ or $g_j(x,A_1) = 0$. Since these are mutually exclusive events (as $x\geq z$), 
\begin{align*}
\E[y_j] &= \Pr[y_j=1/h]/h=(1-(1 - \exp(-n(j,z)h) + \exp(-n(j,x)h)))/h\\
&= (\exp(-n(j,z)h)-\exp(-n(j,x)h))/h.
\end{align*}
Since $n(j,z)\leq n(j,x)$, we know from \nameref{lemS3}, $\E[y_j] \leq  |n(j,x)-n(j,z)|$.

Case 2: $x_j = z_j = 1$. In this case $y_j = |g_j(x,A_1)-g_j(z,A_1)|/h = 0$ a.s. because the recovery time of node $j$ depends only on $A_1$ and not on the state of any other nodes.  Thus the inequality holds trivially. 

Case 3: $x_j = 1$ and $z_j = 0$. Here, $y_j = |g_j(x,A_1) - 1 - g_j(z,A_1)|/h \in \{0,1/h,2/h\}$. Specifically, $y_j = 1/h$ when either $g_j(x,A_1) = 0$ and $g_j(z,A_1) = 0$ OR $g_j(x,A_1) = 1$ and $g_j(z,A_1) = 1$. Thus, \[ \Pr[y_j = 1/h] = (1 - \exp(-\mu h))\exp(-n(j,z) h) + \exp(-\mu h)(1 - \exp(-n(j,z) h)).\] 
Similarly, $y_j = 2/h$ only if $g(x,A_1) = 0$ and $g(z,A_1) = 1$. Thus $\Pr[y_j = 2/h] = (1-\exp(-\mu h))(1 - \exp(-n(j,z)h))$. Thus, 
\begin{align*}
\E[y_j] &= \Pr[y_j = 2/h](2/h) + \Pr[y_j = 1/h](1/h) \\
& = (2-\exp(-\mu h)-\exp(-n(j,z)h))/h \leq \mu + n(j,z)\\
&\leq  |n(j,x) - n(j,z)| + \mu |x_j-z_j|,
\end{align*}
since $z_j=0$, $x_j=1$, and $n(j,x)=0$ by definition.
\end{proof}

\begin{lemma}[Lemma 5]
Euler's method is 0-stable, that is for any two sequences of random variables $(Y_i)$ and $(Z_i)$ adapted to $(F_i)$ with $Y_0=Z_0$,
\[ \E[|Y_i-Z_i|] \leq  K \max_{1\leq j\leq i} \E[|D(Y_j,Y_{j-1},A_j)-D(Z_j,Z_{j-1},A_j)|],\]
where $K=(1/L)(\exp(LT)-1)$, $T=ih$, and $L$ is the Lipschitz constant from Lemma 4 (i.e., \nameref{lem4Aapx} or \nameref{lem4Bapx} depending on the process).
Depending on the process we may write $K$ as $K_{SI}$ or $K_{SIS}$.
\end{lemma}
\begin{proof}
As mentioned in the main body of the paper, the proof follows \cite[p41]{ascher1998computer}. We define $S_i=Y_i-Z_i$ and $\theta = \max_{1\leq j\leq i} \E[|D(Y_j,Y_{j-1},A_j)-D(Z_j,Z_{j-1},A_j)|]$. Thus for all $j$,
\begin{align*}
\theta &\geq  \E[|D(Y_j,Y_{j-1},A_j)-D(Z_j,Z_{j-1},A_j)|] \qquad \text{by definition,}\\
&= \E[|S_j/h - S_{j-1}/h - (f(Y_{j-1},A_j) - f(Z_{j-1},A_j))|] \qquad \text{using the definitions of $D$ and $S$,}\\
&\geq  \E[|S_j|]/h - \E[|S_{j-1}|]/h - \E[|f(Y_{j-1},A_j) - f(Z_{j-1},A_j)|] \qquad\text{using the triangle inequality.}
\end{align*}
Since $A_j$ is distributed like $A_1$ and is independent of $Y_{j-1}$ and $Z_{j-1}$ due to the filtration, we can apply Lemma 4 to obtain, 
$ \theta \geq  \E[|S_j|]/h - \E[|S_{j-1}|]/h - L \E[|S_{j-1}|]$.
Thus, 
$ \E[|S_j|] \leq  \theta h + \E[|S_{j-1}|] (1 + hL)$.
Hence by induction,
\begin{align*}
\E[|S_i|] &\leq  \theta h \sum_{j=1}^i (1 + hL)^{i-j} \qquad \text{since $S_0=0$,}\\
&\leq  \theta (\exp(Lhi)-1)/L.
\end{align*}
\end{proof}
\end{document}